\documentclass[12pt]{amsart}

\newcommand{\elle}[1]{L^{#1}(\Omega)}
\newcommand{\huz}{H^1_0(\Omega)}
\newcommand{\sob}[2]{W^{#1}_{#2}(\Omega)}

\newcommand{\enne}{\mathbb{N}}

\newcommand{\erre}{\mathbb{R}}

\newcommand{\io}{\int\limits_{\Omega}}
\newcommand{\ik}{\int\limits_{\{|\un| \geq k\}}}

\newcommand{\dive}{{\rm div}}
\newcommand{\norma}[2]{\|#1\|_{\lower 4pt \hbox{$\scriptstyle #2$}}}
\newcommand{\fn}{f_{n}}
\newcommand{\un}{u_{n}}

\newcommand{\arrstre}{\renewcommand{\arraystretch}{2}}

\newcommand{\disp}{\displaystyle}
\newcommand{\ba}{\begin{array}}
\newcommand{\ea}{\end{array}}
\newcommand{\be}{\begin{equation}}
\newcommand{\ee}{\end{equation}}

\newcommand{\rife}[1]{(\ref{#1})}
\newtheorem{example}{\sc Example}%[section]

\newtheorem{lemma}[example]{\sc Lemma}

\newtheorem{theo}[example]{\sc Theorem}

\begin{document}

%%%%%%%%%%%%%%%%%%%%%%%%%%%%%%%%%%%%%%%%%%%%%%%%%%%%%%%%%%%%
\title[Noncoercive elliptic problems]{
Existence of solutions for some \\
noncoercive elliptic problems \\
involving
derivatives of nonlinear terms}
%%%%%%%%%%%%%%%%%%%%%%%%%%%%%%%%%%%%%%%%%%%%%%%%%%%%%%%%%%%%
\author{Lucio Boccardo,  Gisella Croce and Luigi Orsina}
%%%%%%%%%%%%%%%%%%%%%%%%%%%%%%%%%%%%%%%%%%%%%%%%%%%%%%%%%%%%
\address{L.B. -- Dipartimento di Matematica, ``Sapienza'' Universit\`{a} di Roma,
P.le A. Moro 2, 00185 Roma (ITALY)}
\email{boccardo@mat.uniroma1.it}
\address{G.C. -- Laboratoire de Math\'ematiques Appliqu\'ees du Havre, Universit\'e du Havre,
25, rue Philippe Lebon, 76063 Le Havre (FRANCE)}
\email{gisella.croce@univ-lehavre.fr}
\address{L.O. -- Dipartimento di Matematica, ``Sapienza'' Universit\`{a} di Roma,
P.le A. Moro 2, 00185 Roma (ITALY)}
\email{orsina@mat.uniroma1.it}
%%%%%%%%%%%%%%%%%%%%%%%%%%%%%%%%%%%%%%%%%%%%%%%%%%%%%%%%%%%%
\begin{abstract}
We study a nonlinear equation with an elliptic ope\-rator having degenerate coercivity. We prove the existence of a  $\sob{1,1}{0}$  solution which is distributional or entropic, according to the growth assumptions on a lower order term in divergence form. 
\end{abstract}

\maketitle

\rightline{\it To Ildefonso:}

\rightline{\it But of all these friends and lovers}

\rightline{\it There is no one compares with you \cite{beatles}}

\section{Introduction and statements of the results}
In a joint paper with Ildefonso Diaz 
 the authors of \cite{bgdm} studied boundary value problems of the type
\be\left\{
%\arrstre
\ba{cl}
\disp
A(u)= f - \dive(\Phi(u)) & \mbox{in $\Omega $,}\\
\hfill u = 0 \hfill & \mbox{on $\partial\Omega$,}
\ea
\right.
\label{1}
\ee
where
\begin{equation}
\label{omega}
\hbox{$\Omega$ is a bounded, open subset of $\erre^{N}$, with $N > 2$,}
\end{equation}
\begin{equation}
\label{coer}
 A\; \hbox{ is a coercive nonlinear differential operator,}
\end{equation}
  acting on $\sob{1,p}{0} $,
$1 < p <\infty$,  defined by 
 $A(v)=-\dive(a(x,v,\nabla v))$, which satisfies the classical Leray-Lions assumptions,
 $$
 f\in\sob{-1,p'}{},  
 $$
\begin{equation}
\label{Phi}
\hbox{$\Phi$ belongs to $C^0(\erre,\erre^N)$.}
\end{equation}
The main feature of  problem (\ref{1}) is that 
no growth assumption was assumed on $\Phi$.

Despite that, the authors proved the existence of a   solution, in the following sense. Let $h\in C^1_c(\erre)$, then $u$ is a renormalized solution to problem (\ref{1}) if  
 \begin{equation}
\label{rino}
\io[a(x,u,\nabla u)-\Phi(u)]\cdot\nabla[h(u)\phi]  
= \io\,f[h(u)\phi] ,
\quad\forall\;\phi\in\mathcal{D}(\Omega).
\end{equation}
In \cite{BCapri} the above problem is studied  under    the  weaker assumption that $f\in L^1(\Omega)$, proving the existence of 
a solution in a slightly different sense.
For $k\geq 0$ and $s\in \erre$, let $T_k(s)=\max\{-k,\min\{s,k\}\}$.
Then $u$ is an entropy solution to (\ref{1}) if $T_k(u)$ belongs to $H^1_0(\Omega)$ for every $k>0$ and for every $\varphi \in H^1_0(\Omega)\cap L^{\infty}(\Omega)$
 \begin{equation}
\label{entro}
 \io[a(x,u,\nabla u)- \Phi(u)]\cdot \nabla T_k(u-\varphi) 
\leq \!\io\! f T_k(u-\varphi).
\end{equation}

In this note we will use the latter approach to  prove the existence of a
 $W^{1,1}_0(\Omega)$  solution to the following
degenerate elliptic  problem:
\be\left\{
\arrstre
\ba{cl}
\disp
-\dive\bigg(\frac{a(x)\,\nabla u}{(1+b(x)|u|)^{2}}\bigg) + u = f - \dive(\Phi(u)) & \mbox{in $\Omega $,}\\
\hfill u = 0 \hfill & \mbox{on $\partial\Omega $.}
\ea
\right.
\label{lineare}
\ee
Here   $a(x),\,b(x)$ are  measurable functions such that
\begin{equation}
\label{ab}
0 <\alpha\leq a(x) \leq\beta,\quad  0 \leq b(x) \leq B,
\end{equation}
  with  $\alpha,\,\beta\in\erre^+$, $B \in \erre$  and
\begin{equation}
\label{f2}
f(x) \hbox{ belongs to } \elle2.
\end{equation}
We point out that the main difference between  the boundary value problems \rife{1} and \rife{lineare} is that the coercivity assumption \rife{coer} is not satisfied by the differential operator in \rife{lineare}.

 We are going to prove that problem \rife{lineare} has a  solution $u$ belonging to the non-reflexive Sobolev space $W^{1,1}_{0}(\Omega)$.
We point out that this is quite unusual for an elliptic problem. According to the growth of $\Phi$, $u$ will be either a distributional or an entropy solution.

We recall that the problems
\be
\left\{
\arrstre
\ba{cl}
\disp
-\dive\bigg(\frac{a(x)\,\nabla u}{(1+|u|)^{\theta}}\bigg) = f & \mbox{in $\Omega $,}\\
\hfill u = 0 \hfill & \mbox{on $\partial\Omega$,}
\ea
\right.
\label{p-teta}
\ee
and 
\be\label{pbBCOlincei}
\left\{
\arrstre
\ba{cl}
\disp
-\dive\bigg(\frac{a(x)\nabla u}{(1+|u|)^{2}}\bigg) + u = f& \mbox{in $\Omega $,}\\
\hfill u = 0 \hfill & \mbox{on $\partial\Omega$,}
\ea
\right.
\ee
have been studied   in \cite{bdo}, \cite{bb}, \cite{croce} and \cite{bco} proving existence results.
In this note we prove that the same results  hold even in the presence of a term in divergence form, that is, for problem (\ref{lineare}).

We are going to prove the following theorems, according to the growth of $\Phi$.
 \begin{theo}\label{thm-lineare}\sl
Assume \rife{omega}, \rife{Phi}, \rife{ab} and \rife{f2} and that there exists a positive $C$ such that
\begin{equation}\label{crescitaPhi}
|\Phi(t)|\leq C\,|t|^{2}\, \quad \forall\,\,t\in \erre.
\end{equation}
Then there exists a distributional
solution $u \in W^{1,1}_0(\Omega)\cap L^2(\Omega)$ to problem (\ref{lineare}), in the sense that
$$
\io \frac{a(x)\,\nabla u\cdot \nabla \varphi}{(1+b(x)|u|)^{2}} + \io u\,\varphi = \io f\,\varphi +\io \Phi(u)\cdot\nabla \varphi\,,
$$
for all $\varphi \in W^{1,\infty}_0(\Omega)$.
\end{theo}
In the case where assumption (\ref{crescitaPhi}) is not satisfied, one can prove the existence of more general solutions,
that is, renormalized solutions as in \cite{bgdm}, or entropy solutions as in \cite{BCapri}. Since the proof of existence of entropy solutions is easier (due to the fact that the weak convergence proved in Lemma \ref{convergenze} is enough),
 we will only prove the second result. Note however that the two concepts of solutions are equivalent (at least in the framework of Lebesgue data, see \cite{DMOP}) so that one can recover the existence of a renormalized solution from the existence of an entropy one.

%%%
\begin{theo}\label{entropysolutions}\sl
Assume \rife{omega}, \rife{Phi}, \rife{ab} and \rife{f2}.
Then there exists an entropy
solution $u \in W^{1,1}_0(\Omega)\cap L^2(\Omega)$ to problem (\ref{lineare}), in the sense that
$T_k(u)$ belongs to $H^1_0(\Omega)$ for every $k>0$ and
$$
\io \frac{a(x)\,\nabla u\cdot \nabla T_k(u-\varphi)}{(1+b(x)|u|)^{2}} + \io u\,T_k(u-\varphi)
$$
$$
\leq \io f\,T_k(u-\varphi)
+\io \Phi(u)\cdot\nabla T_k(u-\varphi)\qquad \forall\,\varphi \in H^1_0(\Omega)\cap L^{\infty}(\Omega)\,.
$$
\end{theo}
%%%%%%%%%%%%%%%%%%%%%%%%%%%%%%%%%%%%%%%%%%%%%%%%%%
%%%%%%%%%%%%%%%%%%%%%%%%%%%%%%%%%%%%%%%%%%%%%%%%%%
%%%%%%%%%%%%%%%%%%%%%%%%%%%%%%%%%%%%%%%%%%%%%%%%%%%%
\section{Proofs of the results}
To prove our existence results, we begin by approximating the boundary value problem \rife{lineare}. Let $\{\fn\}$ be a sequence of $\elle\infty$ functions such that $\fn$ strongly converges to $f$ in $\elle2$, and $|\fn| \leq |f|$ for every $n$ in $\enne$.
\begin{lemma}
There exists
a solution $\un$ in $\huz \cap \elle\infty$ of
\begin{equation}\label{ppn_0}
\left\{
\arrstre
\ba{cl}
\disp
-\dive\bigg(\frac{a(x)\,\nabla \un}{(1+b(x)|\un|)^{2}}\bigg) + \un = \fn-\dive (\Phi(u_n)) & \mbox{in $\Omega $,}\\
\hfill \un = 0 \hfill & \mbox{on $\partial\Omega $.}
\ea
\right.
\end{equation}
\end{lemma}
\begin{proof}
Let $M_n = \norma{f_n}{\elle\infty} + 1$, and consider the problem
\be\label{pbbase}
\left\{
\arrstre
\ba{cl}
\disp
-\dive\bigg(\frac{a(x)\nabla w }{(1+b(x)|T_{M_n}(w)|)^{2}}\bigg) + w = f_n-\dive (\Phi(T_{M_n}(w))) & \mbox{in $\Omega$,} \\
\hfill w = 0 \hfill & \mbox{on $\partial\Omega$.}
\ea
\right.
\ee
The existence of  a $\huz$ weak solution $w$  to \rife{pbbase} follows from Schauder's Theorem. Choosing $(|w| - \norma{f_n}{\elle\infty})_{+}\,{\rm sgn}(w)$ as a test function we obtain, dropping the nonnegative first term,
and using
the divergence Theorem on the last one, that $|w| \leq \norma{f_n}{\elle\infty} < M_n$. Therefore, $T_{M_n}(w) = w$, and $w$ is a bounded weak solution of \rife{ppn_0}.
\end{proof}
In the following result we are going to prove some 
a priori estimates on the solutions $u_n$ to problems
(\ref{ppn_0}).
\begin{lemma}
Let $u_n$ be the sequence of solutions to (\ref{ppn_0}). Then for every $k\geq 0$,
\be\label{aa}
\ik|\un|^{2} \leq \ik|f|^{2}\,;
\ee
\be\label{qq}
\lim_{k \to +\infty}\,{\rm meas}(\{|\un| \geq k\}) = 0\,\,
\mbox{uniformly with respect to $n$;}
\ee
\be\label{bb}
\alpha \io \frac{|\nabla u_n|^2}{(1+B|u_n|)^2}\leq \io |f|^2
\,;
\ee
\be\label{troncate}
\norma{\nabla T_k(u_n)}{L^2(\Omega)}^2\leq \frac{\norma{f}{L^1(\Omega)}}{\alpha}k{(1+Bk)^2}\,.
%\norma{\nabla T_k(u_n)}{L^2(\Omega)}\leq C_k\,,
\ee
%where $C_k$ denotes a constant  depending on $k$, $\alpha$, $f$, independent of $n$.
\end{lemma}
\begin{proof}
Let $k \geq 0$, $i > 0$, and let $\psi_{i,k}(s)$ be the function defined by
$$
\psi_{i,k}(s) =
\left\{
\ba{cl}
0 & \mbox{if $0 \leq s \leq k$,}
\\
i(s-k) & \mbox{if $k < s \leq k + \frac1i$,}
\\
1& \mbox{if $s > k + \frac1i$,}
\\
\psi_{i,k}(s)=-\psi_{i,k}(-s) & \mbox{if $s < 0$.}
\ea
\right.
$$
The choice of $|\un|\,\psi_{i,k}(\un)$ as a test function in \rife{ppn_0} yields
$$
\arrstre
\ba{l}
\disp
\io \frac{a(x)|\nabla\un|^2}{(1 + b(x)|\un|)^{2}} |\psi_{i,k}(\un)|
+
\io \frac{a(x)|\nabla\un|^2}{(1 + b(x)|\un|)^{2}} \psi'_{i,k}(\un) |\un|
\\
\disp
+
\io \un |\un| \psi_{i,k}(\un)
=
\io \fn |\un| \psi_{i,k}(\un)+
\io
\Phi(u_n)\cdot\nabla (|\un|\,\psi_{i,k}(\un))
\,.
\ea
$$
By the divergence Theorem the last term is zero.
Since $\psi'_{i,k}(s) \geq 0$, we can drop the second term of the left hand side.
By (\ref{ab}) one gets
$$
\alpha\io \frac{|\nabla\un|^{2}}{(1 + b(x)|\un|)^{2}}\,|\psi_{i,k}(\un)|
+
\io \un |\un| \psi_{i,k}(\un)
\leq
\io |f| |\un| |\psi_{i,k}(\un)|
\,.
$$
We infer  (\ref{aa}) from this estimate
as in \cite{bco}, letting $i\to \infty$.
One can prove (\ref{qq}) and (\ref{bb}) with the same arguments as in \cite{bco}.

The choice of $T_k(u_n)$  as a test function in (\ref{ppn_0}) gives
$$
\frac{\alpha}{(1+Bk)^2}\norma{\nabla T_k(u_n)}{L^2(\Omega)}^2\leq k \,\norma{f}{L^1(\Omega)} + \io \Phi(u_n)\cdot \nabla T_k(u_n)
$$
by using (\ref{ab}) and dropping the positive term $\displaystyle \io u_nT_k(u_n)$.
By the divergence Theorem  the last integral is zero.
This implies (\ref{troncate}).
\end{proof}
The estimates proved in Lemma 4 can be used as
in \cite{bco} to prove the following result.
\begin{lemma}\label{convergenze}
Let $u_n$ be the solutions to (\ref{ppn_0}). Up to subsequences, the sequence $\{u_n\}$ converges to some function $u$
strongly  in $\elle2$ and weakly in $W^{1,1}_0(\Omega)$.
\end{lemma}

We are going to prove Theorem \ref{thm-lineare}.
\begin{proof}
Let $u_n$ and $u$ be as in Lemma \ref{convergenze}.
We now pass to the limit in the approximate problems \rife{ppn_0}.
The lower order term on the left hand side and the first term of right hand side easily pass to the limit,
due to the $L^2(\Omega)$ convergence of $\un$ to $u$ and of $\fn$ to $f$.
For the operator term one can pass to the limit as in \cite{bco}.

For the last term, since $u_n$ converges to $u$
in $L^{2}(\Omega)$ and thus a.e. in $\Omega$,  and $\Phi$ is continuous, $\Phi(u_n)\to \Phi(u)$ a.e. in $\Omega$.
Moreover, if $E$ is any measurable subset of $\Omega$ we have, by (\ref{crescitaPhi}),
$$
\int\limits_{E}|\Phi(u_n)|\leq  C \int\limits_{E}|u_n|^2\,.
$$
The last term tends to 0, as ${\rm meas}(E)\to 0$, uniformly with respect to $n$, by Vitali's Theorem.
 Again by Vitali's Theorem, we deduce that $\Phi(u_n)\to \Phi(u)$
in $(L^1(\Omega))^N$. This allows us to pass to the limit in the last term.
\end{proof}
We are now going to prove Theorem \ref{entropysolutions}.
\begin{proof}
We consider $T_k(u_n-\varphi)$ as a test function in (\ref{ppn_0}) and we pass to the limit as $n\to \infty$.
We can write the operator term as
$$
\io \frac{a(x)}{(1+b(x)|u_n|)^2}|\nabla T_k(u_n-\varphi)|^2 +
\io \frac{a(x)}{(1+b(x)|u_n|)^2}\nabla \varphi\cdot \nabla T_k(u_n-\varphi)\,.
$$
Estimate (\ref{troncate}) and the a.e. convergence of $u_n$ to $u$ imply that $T_k(u_n-\varphi)\to T_k(u-\varphi)$ weakly in $H^1_0(\Omega)$. Since $\displaystyle \frac{a(x)}{(1+b(x)|u_n|)^2}$ is bounded in $\Omega$, we deduce  that
$$
\liminf\limits_{n\to \infty}
\io \frac{a(x)|\nabla T_k(u_n-\varphi)|^2}{(1+b(x)|u_n|)^2}
\geq
\io \frac{a(x)|\nabla T_k(u-\varphi)|^2}{(1+b(x)|u|)^2}\,.
$$
For the second term one has
$$
\io \frac{a(x)}{(1+b(x)|u_n|)^2}\nabla \varphi\cdot \nabla T_k(u_n-\varphi)\to
\io \frac{a(x)}{(1+b(x)|u|)^2}\nabla \varphi\cdot \nabla T_k(u-\varphi)
$$
since $T_k(u_n-\varphi)\to T_k(u-\varphi)$ weakly in $H^1_0(\Omega)$
and $\displaystyle\frac{a(x)}{(1+b(x)|u_n|)^2}\nabla \varphi\to \frac{a(x)}{(1+b(x)|u|)^2}\nabla \varphi$ in $(L^2(\Omega))^N$
by Lebesgue's Theorem.

By the  $L^2(\Omega)$ convergences of $u_n$ to $u$  and $f_n$ to $f$ we deduce that
$$
\io u_n\,T_k(u_n-\varphi)\to \io u\,T_k(u-\varphi)\,,\qquad
\io f_n\,T_k(u_n-\varphi)\to \io f\,T_k(u-\varphi)\,.
$$
Let us now study the last term:
$
\displaystyle \io \Phi(u_n)\cdot \nabla T_k(u_n-\varphi)
$. This is non zero only in $\{|u_n-\varphi|\leq k\}$. On this set $\Phi(u_n)$ is bounded, by the continuity of $\Phi$.
By the weak $H^1_0(\Omega)$ convergence of $T_k(u_n-\varphi)$ to $T_k(u_n-\varphi)$ we deduce that
$$
\io \Phi(u_n)\cdot \nabla T_k(u_n-\varphi)\to
\io \Phi(u)\cdot \nabla T_k(u-\varphi)\,,
$$
as desired.
\end{proof}
%%%%%%%%%%%%%%%%%%%%%%%%%%%%%%%%%%%%%%%%%%%%%
\section*{acknowledgments}
This work contains
the unpublished part of the lecture
({\sl Dirichlet problems with singular
convection terms and applications 
 with and without  
  Ildefonso}. Toledo, 
  June 16, 2011.)
 of the first author at the Conference
``Nonlinear Models in Partial Differential Equations
(An international congress on occasion of Jesus Ildefonso Diaz's 60th birthday)''
%%%%%%%%%%%%%%%


\begin{thebibliography}{99}
%%%%%%%%%%%%%%%%%%%%%%%%%%%%%%%%%%%%%%%%%%%%%%%%%%%%%%%%%%%%



\bibitem{BBGGPV}
Ph. B{\'e}nilan, L.~Boccardo, T.~Gallou{\"e}t, R.~Gariepy, M.~Pierre and J.L.
  V{\'a}zquez:
{\sl An {$L\sp 1$}-theory of existence and uniqueness of solutions of
  nonlinear elliptic equations}.
Ann. Scuola Norm. Sup. Pisa Cl. Sci. {\bf 4} (1995), 241--273.

 


\bibitem{BCapri}
L. Boccardo:
{\sl Some nonlinear Dirichlet problems in $L^1$ involving lower order terms in divergence form}.
Progress in elliptic and parabolic partial differential equations (Capri, 1994), 43–57,
Pitman Res. Notes Math. Ser., 350, Longman, Harlow, 1996. 


\bibitem{bb}
L. Boccardo, H. Brezis:
{\sl Some remarks on a class of elliptic equations}.
Boll. Unione Mat. Ital.
{\bf 6} (2003), 521--530.

\bibitem{bco}
L. Boccardo, G. Croce, L. Orsina:
{\sl
Nonlinear degenerate elliptic problems with $W^{1,1}_{0}$ solutions}, to appear in {\it Manuscripta Math.}.




\bibitem{bgdm}
L. Boccardo, D. Giachetti, J. I.  Diaz, F. Murat:  {\sl Existence and regularity of renormalized solutions for some elliptic problems involving derivatives of nonlinear terms}. J. Differential Equations 106 (1993),  215--237.


\bibitem{bdo}
L. Boccardo, A. Dall'Aglio, L. Orsina:
{\sl Existence and regularity results for some elliptic equations with degenerate coercivity}, dedicated to Prof. C. Vinti (Perugia, 1996).
Atti Sem. Mat. Fis. Univ. Modena
{\bf 46} (1998), 51--81.


\bibitem{croce}
G. Croce:
{\sl The regularizing effects of some lower order terms in an elliptic equation with degenerate coercivity}.
Rendiconti di Matematica
{\bf 27} (2007), 299--314.

\bibitem{DMOP} G. Dal Maso, F. Murat, L. Orsina, A. Prignet: {\sl Renormalized solutions for elliptic equations with general measure data}, Ann. Scuola Norm. Sup. Pisa Cl. Sci., {\bf 28} (1999), 741--808.


%\bibitem{Po}
%A. Porretta:
%{\sl Uniqueness and homogenization for a class of noncoercive operators in divergence form}, dedicated to Prof. C. Vinti (Perugia, 1996).
%Atti Sem. Mat. Fis. Univ. Modena
%{\bf 46} (1998), 915--936.

%\bibitem{st}
%G. Stampacchia:
%{\sl Le probl\`eme de Dirichlet pour les \'equations elliptiques du second ordre \`a coefficients discontinus}.
%Ann. Inst. Fourier (Grenoble)
%{\bf 15} (1965), 189--258.

\bibitem{beatles}
 J. Lennon, P. McCartney: 
{\sl In my life}. Rubber Soul (1965).

\end{thebibliography}
\end{document}